\documentclass [a4paper]{article}

\usepackage[utf8]{inputenc}

\usepackage{amsmath,amsfonts, amsthm, amssymb, version}
\usepackage[colorlinks=false,urlcolor=blue,linkcolor=green]{hyperref}

\numberwithin{equation}{section}

\allowdisplaybreaks[1]

\newtheorem{theorem}{Theorem}
\newtheorem{corollary}[theorem]{Corollary}
\newtheorem{proposition}[theorem]{Proposition}
\newtheorem{lemma}[theorem]{Lemma}
\newtheorem*{theorem*}{Theorem}
\newtheorem*{proposition*}{Proposition}
\newtheorem*{lemma*}{Lemma}

\theoremstyle{remark}

\theoremstyle{definition}

\newcommand{\N}{\mathbb{N}}
\newcommand{\R}{\mathbb{R}}
\newcommand{\C}{\mathbb{C}}
\newcommand{\K}{\mathbb{K}}

\newcommand\norm[1]{\Vert #1 \Vert}

\numberwithin{theorem}{section}

\begin{document}

\title{A note on operator tuples which are $(m,p)$-isometric as well as $(\mu,\infty)$-isometric}
\author{Philipp H. W. Hoffmann}

\date{}
\maketitle

\renewcommand{\thefootnote}{\fnsymbol{footnote}}
\begin{center}
\emph{This is a preprint of an article which has appeared in:\\ Operators and Matrices, Volume 11, Number 3 (2017), 623-633\footnote{The final publication is available under \url{http://oam.ele-math.com/}}
\\$\textrm{}$\\}
\end{center}

\begin{abstract}
We show that if a tuple of commuting, bounded linear operators
$(T_1,...,T_d) \in B(X)^d$ is both an $(m,p)$-isometry and a $(\mu,\infty)$-isometry, then the tuple $(T_1^m,...,T_d^m)$ is a 
$(1,p)$-isometry. We further prove some additional properties of the operators $T_1,...,T_d$ and show a stronger result in the case of a
commuting pair $(T_1,T_2)$. 
\end{abstract}

\maketitle



\section{Introduction}
\setcounter{page}{1}

Let in the following $X$ be a normed vector space over $\K \in \{\R, \C\}$ and let the symbol $\N$ denote the natural numbers including $0$.  

A tuple of commuting linear operators $T:=(T_1,...,T_d)$ with $T_j:X \to X$ is called an \emph{$(m,p)$-isometry} (or an \emph{$(m,p)$-isometric tuple}) if, and only if, for given $m \in \N$ and $p \in (0,\infty)$,
\begin{align}\label{def. (m,p)-tuples}
\sum_{k=0}^m (-1)^{m-k} \binom{m}{k} \sum_{|\alpha|=k} \frac{k!}{\alpha!} \norm{T^\alpha x}^p = 0, \ \ \forall x \in X.
\end{align}
Here, $\alpha:=(\alpha_1,...,\alpha_d) \in \N^d$ is a multi-index, $|\alpha|:=\alpha_1 + \cdots + \alpha_d$ the sum of its entries, $\frac{k!}{\alpha!} := \frac{k!}{\alpha_1! \cdots \alpha_d!}$ a multinomial coefficient and $T^\alpha:=T_1^{\alpha_1} \cdots T_d^{\alpha_d}$, where $T_j^0:=I$ is the identity operator.

Tuples of this kind have been introduced by Gleason and Richter \cite{GleaRi} on Hilbert spaces (for $p=2$) and have been further studied on 
general normed spaces in \cite{HoMa}. The tuple case generalises the single operator case, originating in the works of Richter \cite{Ri} and 
Agler \cite{Ag} in the 1980s and being comprehensively studied in the Hilbert space case by Agler and Stankus \cite{AgStanI}; the single 
operator case on Banach spaces has been introduced in the case $p=2$ in \cite{Bo} and \cite{SidAh} and in its general form by Bayart in \cite{Bay}. We remark that boundedness, although usually assumed, is not essential for the definition of $(m,p)$-isometries, as shown by Bermúdez, Martinón and Müller in \cite{BeMaMu}. Boundedness does, however, play an important role in the theory of objects of the following kind:

Let $B(X)$ denote the algebra of bounded (i.e. continuous) linear operators on $X$. Equating sums over even and odd $k$ and then considering $p \to \infty$ in \eqref{def. (m,p)-tuples}, leads to the definition of \emph{$(m,\infty)$-isometries} (or 
\emph{$(m,\infty)$-isometric tuples}). That is, a tuple of commuting, bounded linear operators $T \in B(X)^d$ is referred to as an 
$(m,\infty)$-isometry if, and only if, for given $m \in \N$ with $m \geq 1$,
\begin{align}\label{def. (mu,infty)-tuples}
\max_{\substack{|\alpha|=0,...,m \\ |\alpha| \ \textrm{even}}} \norm{T^\alpha x} 
= \max_{\substack{|\alpha|=0,...,m \\ |\alpha| \ \textrm{odd}}} \norm{T^\alpha x}, \ \ \forall x \in X.
\end{align}     

These tupes have been introduced in \cite{HoMa}, with the definition of the single operator case appearing in \cite{HoMaOS}. Although, it is known that tuples containing unbounded operators exist which satisfy equation \eqref{def. (mu,infty)-tuples}, several important statements on 
$(m,\infty)$-isometries require boundedness. Therefore, from now on, we will always assume the operators $T_1,...,T_d$ to be bounded.

In \cite{HoMa}, the question is asked what necessary properties a commuting tuple $T \in B(X)^d$ has to satisfy if it is both an $(m,p)$-isometry and a $(\mu,\infty)$-isometry, where possibly $m \neq \mu$. In the single operator case this question is trivial and answered in \cite{HoMaOS}: If $T=T_1$ is a single operator, then the condition that $T_1$ is an $(m,p)$-isometry is equivalent to the mappings $n \mapsto \norm{T_1^nx}^p$ being polynomial of degree $\leq m-1$ for all $x \in X$. This has been already been observed for operators on Hilbert spaces in \cite{GleaRi} and shown in the Banach space/normed space case in \cite{HoMaOS}; the necessity of the mappings $n \mapsto \norm{T_1^nx}^p$ being polynomial has also been proven in \cite{Bay} and \cite{BeMaNe}. On the other hand, in \cite{HoMaOS} it is shown that if a bounded operator $T=T_1 \in B(X)$ is a $(\mu,\infty)$-isometry, then the mappings $n \mapsto \norm{T_1^nx}$ are bounded for all $x \in X$. The conclusion is obvious: if $T = T_1 \in B(X)$ is both $(m,p)$- and $(\mu,\infty)$-isometric, then for all $x \in X$ the $n \mapsto \norm{T_1^nx}^p$ are always constant and $T_1$ has to be an isometry (and, since every isometry is $(m,p)-$ and $(\mu,\infty)$-isometric, we have equivalence).

The situation is, however, far more difficult in the multivariate, that is, in the operator tuple case. Again, we have equivalence between $T=(T_1,...,T_d)$ being an $(m,p)$-isometry and the mappings $n \mapsto \sum_{|\alpha|=n}\frac{n!}{\alpha!}\norm{T^\alpha x}^p$ being polynomial of degree $\leq m-1$ for all $x \in X$. The necessity part of this statement has been proven in the Hilbert space case in \cite{GleaRi} and equivalence in the general case has been shown in \cite{HoMa}. On the other hand, one can show that if $T \in B(X)^d$ is a $(\mu,\infty)$-isometry, then the families $\left(\norm{T^\alpha x}\right)_{\alpha \in \N^d}$ are bounded for all $x \in X$, which has been proven in \cite{HoMa}. But this fact only implies that the polynomial growth of the $n \mapsto \sum_{|\alpha|=n}\frac{n!}{\alpha!}\norm{T^\alpha x}^p$ has to caused by the factors $\frac{n!}{\alpha!}$ and does not immediately give us any further information about the tuple $T$.  

There are several results in special cases proved in \cite{HoMa}. For instance, if a commuting tuple $T=(T_1,...,T_d) \in B(X)^d$ is an $(m,p)$-isometry as well as a $(\mu,\infty)$-isometry and we have $m=1$ or $\mu=1$ or $m=\mu=d=2$, then there exists one operator $T_{j_0} \in \{T_1,...,T_d\}$ which is an isometry and the remaining operators $T_k$ for $k \neq j_0$ are in particular nilpotent of order $m$. Although, we are not able to obtain such a results for general $m \in \N$ and $\mu, d \in \N \setminus \{0\}$, yet, we can prove a weaker property: In all proofs of the cases discussed in \cite{HoMa}, the fact that the tuple $(T^m_1,...,T^m_d)$ is a $(1,p)$-isometry is of critical importance (see the proofs of Theorem 7.1 and Proposition 7.3 in \cite{HoMa}). We will show in this paper that this fact holds in general for any tuple which is both $(m,p)$-isometric and $(\mu,\infty)$-isometric, for general $m$, $\mu$ and $d$. 

The notation we will be using is basically standard, with one possible exception: We will denote the tuple of $d-1$ operators obtained by removing one operator $T_{j_0}$ from $(T_1,...,T_d)$ by $T'_{j_0}$, that is $T'_{j_0}:=(T_1,...,T_{j_0-1},T_{j_0+1},...,T_d) \in B(X)^{d-1}$  (not to be confused with the dual of the operator $T_{j_0}$, which will not appear in this paper). Analogously, we denote by $\alpha'_{j_0}$ the multi-index obtained by removing $\alpha_{j_0}$ from $(\alpha_1,...,\alpha_d)$. 
We will further use the notation $N(T_j)$ for the kernel (or nullspace) of an operator $T_j$.

\section{Preliminaries}

In this section, we introduce two needed definitions/notations and compile a number of propositions and theorems from
\cite{HoMa}, which are necessary for our considerations. 

In the following, for $T \in B(X)^d$ and given $p \in (0,\infty)$, define for all $x \in X$ the sequences
$(Q^{n,p}(T,x))_{n \in \N}$ by 
\begin{equation}\label{def. of Q^(n,p)}
Q^{n,p}(T,x):=\sum_{|\alpha|=n}\frac{n!}{\alpha!}\|T^{\alpha} x\|^p.
\end{equation} 
Define further for all $\ell \in \N$ and all $x \in X$, the mappings $P^{(p)}_{\ell}(T, \cdot):X \rightarrow \R$, by 
\begin{align}\label{def. of P^((p))} \notag
	P^{(p)}_\ell(T,x):= &\sum_{k=0}^\ell (-1)^{\ell-k} \binom{\ell}{k} Q^{k,p}(T,x)\\ 
	= &\sum_{k=0}^\ell (-1)^{\ell-k} \binom{\ell}{k} \sum_{|\alpha|=k} \frac{k!}{\alpha!}\norm{T^\alpha x}^p.
\end{align}
It is clear that $T \in B(X)^d$ is an $(m,p)$-isometry if, and only if, $P^{(p)}_m(T, \cdot ) \equiv 0$.

If the context is clear, we will simply write $P_{\ell}(x)$ and $Q^n(x)$ instead of $P^{(p)}_{\ell}(T, x)$ and 
$Q^{n,p}(T,x)$.

Further, for $n,k \in \N$, define the (descending) Pochhammer symbol $n^{(k)}$ as follows: 
\begin{align*}
	n^{(k)}:= \left\{ \begin{array}{ll}
					0, &\textrm{if} \ k > n,\\[1ex]
       				\binom{n}{k}k!\ , \  &\textrm{else}.
       				\end{array} \right.
\end{align*}
Then $n^{(0)} = 0^{(0)} = 1$ and, if $n,k>0$ and $k \leq n$, we have 
\begin{equation*}
n^{(k)} = n (n-1) \cdots (n-k+1).
\end{equation*}

As mentioned above, a fundamental property of $(m,p)$-isometries is that their defining property can be expressed in terms of polynomial sequences.

\begin{theorem}[{\cite[Theorem 3.1]{HoMa}}]\label{Theorem fundamental prop. (m,p)}
$T \in B(X)^d$ is an $(m,p)$-isometry if, and only if, there exists a family of polynomials 
$f_x:\R \rightarrow \R$, $x \in X$, of degree $\leq m-1$ with $f_x|_{\N}= (Q^{n}(x))_{n \in \N}$. \footnote{\label{footnote}Set $\deg 0 := -\infty$ to account for the case $m=0$.}
\end{theorem}

The following statement describes the Newton-form of the Lagrange-polynomial $f_x$ interpolating $(Q^{n}(x))_{n \in \N}$. 

\begin{corollary}[{\cite[Proposition 3.2.(i)]{HoMa}}]\label{Coro. Newtonform of (m,p) (i) and P_(m-1) (ii)}
Let $m \geq 1$ and $T \in B(X)^d$ be an $(m,p)$-isometry. Then we have for all $n \in \N$
\begin{equation*}
Q^{n}(x)=\sum\limits_{k=0}^{m-1}n^{(k)}\left(\frac{1}{k!}P_k(x)\right), \ \ \forall x \in X.
\end{equation*} 
\end{corollary}

Regarding $(m,\infty)$-isometries, we will need the following two statements. Theorem \ref{Th. fundamental prop. and partition for (m,infty)} is a combination of several fundamental properties of $(m,\infty)$-isometric tuples. 

\begin{theorem}[{\cite[Corollary 5.1]{HoMa}}]\label{Prop. norm(T^alpha x) is bounded}
Let $T=(T_1,...,T_d) \in B(X)^d$ be an\\ 
$(m,\infty)$-isometry. Then $(\norm{T^{\alpha}x})_{\alpha \in \N^d}$ is bounded, for all $x \in X$, and
\begin{equation*}
\max_{\alpha \in \N^d} \norm{T^{\alpha}x} = \max_{|\alpha|= 0,...,m-1} \norm{T^{\alpha}x},
\end{equation*}
for all $x \in X$.
\end{theorem}

\begin{theorem}[{\cite[Proposition 5.5, Theorem 5.1 and Remark 5.2]{HoMa}}]\label{Th. fundamental prop. and partition for (m,infty)}
Let \\ 
$T=(T_1,...,T_d) \in B(X)^d$ be an $(m,\infty)$-isometric tuple. Define the norm
$|.|_{\infty}: X \to [0,\infty)$ via $|x|_{\infty}:= \max_{\alpha \in \N^d} \norm{T^{\alpha}x}$, for all 
$x \in X$, and denote
\begin{equation*}
X_{j, |.|_\infty} := \{ x \in X \ | \ |x|_\infty = |T_j^n x|_\infty \ \textrm{for all} \ n \in \N \}.
\end{equation*}
Then 
\begin{equation*}
X = \bigcup_{j=1,...,d} X_{j,|.|_\infty}.
\end{equation*}
\end{theorem}

(Note that, by Theorem \ref{Prop. norm(T^alpha x) is bounded}, $|.|_\infty = \norm{.}$ if $m=1$.)

We will also require a fundamental fact on tuples which are both $(m,p)$- and $(\mu,\infty)$-isometric and an (almost) immediate corollary.

\begin{lemma}[{\cite[Lemma 7.2]{HoMa}}]\label{Lemma T_i^mT_j^m=0}
Let $T =(T_1,...,T_d) \in B(X)^d$ be an\\ 
$(m,p)$-isometry as well as a $(\mu, \infty)$-isometry. Let 
$\gamma=(\gamma_1,...,\gamma_d) \in \N^d$ be a multi-index with the property that
$|\gamma'_j| \geq m$ for every $j \in \{1,...,d\}$. Then $T^{\gamma}=0.$ 
\end{lemma}

Conversely, this implies that if an operator $T^\alpha$ is not the zero-operator, the multi-index $\alpha$ has to be of a specific form. The proof in \cite{HoMa} of the following corollary appears to be overly complicated, the statement is just the negation of the previous lemma.

\begin{corollary}[{\cite[Corollary 7.1]{HoMa}}]\label{Cor T_i^mT_j^m=0}
Let $T =(T_1,...,T_d) \in B(X)^d$ be an 
$(m,p)$-isometry for some $m \geq 1$ as well as a $(\mu, \infty)$-isometry. 
If $\alpha \in \N^d$ is a multi-index with $T^{\alpha} \neq 0$ and $|\alpha|=n$, then there exists some $j_0 \in \{1,...,d\}$ with
$T^{\alpha}=T_{j_0}^{n-|\alpha'_{j_0}|}(T'_{j_0})^{\alpha'_{j_0}}$ and $|\alpha'_{j_0}|\leq m-1$.
\end{corollary}

This fact has consequences for the appearance of elements of the sequences\\ 
$\left(Q^n(x)\right)_{n \in \N}$, since several summands become zero for large enough $n$. That is, we have trivially by definition \eqref{def. of Q^(n,p)} of $\left(Q^n(x)\right)_{n\in \N}$:

\begin{corollary}[{see \cite[proof of Theorem 7.1]{HoMa}}]\label{Cor. structure of Q^n if (m,p) and (mu,infty)}
Let $T =(T_1,...,T_d) \in B(X)^d$\\ 
be an $(m,p)$-isometry for some $m \geq 1$ as well as a $(\mu, \infty)$-isometry. Then, for all $n \in \N$ with $n \geq 2m-1$, we have
\begin{equation*}
 Q^n(x)= \sum_{\substack{\beta \in \N^{d-1} \\ |\beta|=0,...,m-1}} \sum_{j=1}^d 
	\frac{n!}{(n-|\beta|)!\beta!} \norm{T_j^{n-|\beta|}(T'_j)^{\beta}x}^p,
 	\ \ \forall x \in X,
\end{equation*}
where $\frac{n!}{(n-|\beta|)!\beta!}= \frac{n^{(|\beta|)}}{\beta!}$. (We set $n \geq 2m-1$ to ensure that every multi-index only appears once.)
\end{corollary}

\section{The main result}

We first present the main result of this article, which is a generalisation of \cite[Proposition 7.3]{HoMa}, before stating a preliminary lemma needed for its proof.

\begin{theorem}\label{main theorem of paper}
Let $T=(T_1,...,T_d) \in B(X)^d$ be an $(m,p)$-isometric as well a $(\mu,\infty)$-isometric tuple. Then 
\begin{itemize}
\item[(i)] the sequences $n \mapsto \norm{T_j^nx}$ become constant for $n \geq m$, for all $j \in \{1,...,d\}$, for all $x \in X$.
\item[(ii)] the tuple $(T_1^m,...,T_d^m)$ is a $(1,p)$-isometry, that is
\[\sum_{j=1}^d \norm{T_j^{m}x}^p = \norm{x}^p, \ \ \ \forall x \in X.\]
\item[(iii)] for any $(n_1,...,n_d) \in \N^d$ with $n_j \geq m$ for all $j$, the operators 
$\sum_{j=1}^dT_j^{n_j}$ are isometries, that is
\[ \left \|\sum_{j=1}^d T_j^{n_j} x \right \| = \norm{x}, \ \ \ \forall x \in X.\]
\end{itemize}
\end{theorem}

Of course, (i) and (ii) imply that, for any $(n_1,...,n_d) \in \N^d$ with $n_j \geq m$ for all $j$, 
\begin{equation*}
\sum_{j=1}^d \norm{T_j^{n_j}x}^p = \norm{x}^p, \ \ \ \forall x \in X,
\end{equation*}

Theorem \ref{main theorem of paper} is a consequence of the following lemma, which is a weaker version of \ref{main theorem of paper}.(i).

\begin{lemma}\label{lemma for main theorem}
Let $T=(T_1,...,T_d) \in B(X)^d$ be an $(m,p)$-isometric as well as a $(\mu,\infty)$-isometric tuple. Let 
further $\kappa \in \N^{d-1}$ be a multi-index with $|\kappa| \geq 1$. Then the mappings
\begin{align*}
	n \mapsto \norm{T^n_{j}\left(T_{j}'\right)^{\kappa}x}
\end{align*}
become constant for $n \geq m$, for all $j \in \{1,...,d\}$, for all $x \in X$.
\end{lemma}
\begin{proof}
If $m=0$, then $X=\{0\}$ and if $m=1$, the statement holds trivially, since $T_jT_i=0$ for all $i \neq j$ by Lemma \ref{Lemma T_i^mT_j^m=0}. So assume $m \geq 2$. Further, it clearly suffices to consider $|\kappa| = 1$, since the statement then holds for all $x \in X$. The proof, however, works by proving the theorem for $|\kappa| \in \{1,...,m-1\}$ in descending order. (Note that the case $|\kappa| \geq m$ is also trivial, again by Lemma \ref{Lemma T_i^mT_j^m=0}.)

Now fix an arbitrary $j_0 \in \{1,...,d\}$, let $\kappa \in \N^{d-1}$ with $|\kappa| \in \{1,...,m-1\}$ and set $\ell := m - |\kappa|$. Then $\ell \in \{1,...,m-1\}$ and $|\kappa| = m-\ell$. We apply Lemma \ref{Lemma T_i^mT_j^m=0} to $Q^k(T^m_{j_0}\left(T_{j_0}'\right)^{\kappa}x)$.

By definition 
\eqref{def. of Q^(n,p)},
\begin{align}\label{sum up to ell-1} \notag
&Q^k(T^m_{j_0}\left(T_{j_0}'\right)^{\kappa}x)=\sum_{|\alpha|=k}\frac{k!}{\alpha!}\norm{T^{\alpha}\left(T^m_{j_0}\left(T_{j_0}'\right)^{\kappa}x\right)}^p \\ \notag
&= \norm{T_{j_0}^{k}\left(T^m_{j_0}\left(T_{j_0}'\right)^{\kappa}x\right)}^p + 
\sum_{j=1}^k \sum_{\substack{\beta \in \N^{d-1} \\ |\beta| = j}}
\frac{k!}{(k-j)!\beta!}\norm{T_{j_0}^{k-j}\left(T'_{j_0}\right)^{\beta}\left(T^m_{j_0}\left(T_{j_0}'\right)^{\kappa}x\right)}^p \\ \notag
&\overset{\ref{Lemma T_i^mT_j^m=0}}{=} \norm{T_{j_0}^{m+k}\left(T_{j_0}'\right)^{\kappa}x}^p + 
\sum_{j=1}^{\min\{k,\ell-1\}} \sum_{\substack{\beta \in \N^{d-1} \\ |\beta| = j}}
\frac{k!}{(k-j)!\beta!}\norm{T_{j_0}^{m+k-j}\left(T_{j_0}'\right)^{\kappa+\beta}x}^p \\
&= \norm{T_{j_0}^{m+k}\left(T_{j_0}'\right)^{\kappa}x}^p + 
\sum_{j=1}^{\ell-1} k^{(j)}  \sum_{\substack{\beta \in \N^{d-1} \\ |\beta| = j}}
\frac{1}{\beta!}\norm{T_{j_0}^{m+k-j}\left(T_{j_0}'\right)^{\kappa+\beta}x}^p,
\end{align} 
for all $k \in \N$, for all $x \in X$. Here, in the last line, we utilise the fact that $k^{(j)}=0$ if $j>k$.

We now prove our statement by (finite) induction on $\ell$.\\ 

\underline{$\ell=1$:}

For $\ell =1$ and $|\kappa| = m-1$, we have, by \eqref{sum up to ell-1},
\begin{align*}
Q^k\left(T^m_{j_0}\left(T_{j_0}'\right)^{\kappa}x\right)
=\norm{T^{m+k}_{j_0}\left(T_{j_0}'\right)^{\kappa}x}^p, \ \ \forall k \in \N, \ \forall x \in X.
\end{align*}
Since we know by Theorem \ref{Theorem fundamental prop. (m,p)} that the sequences $k \mapsto Q^k\left(T^m_{j_0}\left(T_{j_0}'\right)^{\kappa}x\right)$ are polynomial for all $x \in X$, and by Theorem \ref{Prop. norm(T^alpha x) is bounded} that the $k \mapsto \norm{T^{m+k}_{j_0}\left(T_{j_0}'\right)^{\kappa}x}^p$ are bounded for all $x \in X$, it follows that
\begin{align*}
	n \mapsto \norm{T^n_{j_0}\left(T_{j_0}'\right)^{\kappa}x}
\end{align*}
become constant for $n \geq m$, for all $x \in X$.

Since $\ell \in \{1,...,m-1\}$, if we have $m=2$, we are already done. So assume in the following that $m \geq 3$. \\

\underline{$\ell \rightarrow \ell+1$:}

Assume that the statement holds for some $\ell \in \{1,...,m-2\}$. That is, for all $\kappa \in \N^{d-1}$ with
$|\kappa| = m - \ell$ the sequences
\begin{align*}
	n \mapsto \norm{T^n_{j_0}\left(T_{j_0}'\right)^{\kappa}x}
\end{align*}
become constant for $n \geq m$, for all $x \in X$.

Now take a multi-index $\tilde{\kappa} \in \N^{d-1}$ with $|\tilde{\kappa}| = m- (\ell+1)$ and consider
\begin{align*}
&Q^k(T^m_{j_0}\left(T_{j_0}'\right)^{\tilde{\kappa}}x)= \norm{T_{j_0}^{m+k}\left(T_{j_0}'\right)^{\tilde{\kappa}}x}^p + 
\sum_{j=1}^{\ell} k^{(j)}  \sum_{\substack{\beta \in \N^{d-1} \\ |\beta| = j}}
\frac{1}{\beta!}\norm{T_{j_0}^{m+k-j}\left(T_{j_0}'\right)^{\tilde{\kappa}+\beta}x}^p. 
\end{align*}
(Where we are now summing over all $j$ running from $1$ to $(\ell + 1) - 1 = \ell$.) 

Since $|\beta| \geq 1$, we have $|\tilde{\kappa}+\beta| \geq m-\ell$. Hence, if $k \geq j$, by our induction assumption, 
\begin{align*}
	\norm{T_{j_0}^{m+k-j}\left(T_{j_0}'\right)^{\tilde{\kappa}+\beta}x}^p 
	= \norm{T_{j_0}^{m}\left(T_{j_0}'\right)^{\tilde{\kappa}+\beta}x}^p, \ \ \forall x \in X,
\end{align*}
since $n \mapsto \norm{T_{j_0}^{n}\left(T_{j_0}'\right)^{\tilde{\kappa}+\beta}x}$ become constant for $n \geq m$.

Hence, we have, for all $x \in X$,
\begin{align}\label{quasi-polynomial for k}
Q^k(T^m_{j_0}\left(T_{j_0}'\right)^{\tilde{\kappa}}x) 
= \norm{T_{j_0}^{m+k}\left(T_{j_0}'\right)^{\tilde{\kappa}}x}^p + 
\sum_{j=1}^{\ell} k^{(j)}  \sum_{\substack{\beta \in \N^{d-1} \\ |\beta| = j}}
\frac{1}{\beta!}\norm{T_{j_0}^{m}\left(T_{j_0}'\right)^{\tilde{\kappa}+\beta}x}^p.
\end{align}
That is, for all $x \in X$, the sequences $k \mapsto Q^k(T^m_{j_0}\left(T_{j_0}'\right)^{\tilde{\kappa}}x)$ become almost polynomial (of degree $\leq \ell$), with the term $\norm{T_{j_0}^{m+k}\left(T_{j_0}'\right)^{\tilde{\kappa}}x}^p$ instead of a (constant) trailing coefficient.

But, as before, by Theorem \ref{Theorem fundamental prop. (m,p)}, we know that for any $x \in X$, the sequences $k \mapsto Q^k(T_{j_0}^{m}\left(T_{j_0}'\right)^{\tilde{\kappa}}x)$ are indeed polynomial. Through Corollary \ref{Coro. Newtonform of (m,p) (i) and P_(m-1) (ii)} we know that their trailing coefficients are $\norm{T^m_{j_0}\left(T_{j_0}'\right)^{\tilde{\kappa}}x}^p$. 
Since, by Theorem \ref{Prop. norm(T^alpha x) is bounded}, for each $x \in X$, the sequences $k \mapsto \norm{T_{j_0}^{m+k}\left(T_{j_0}'\right)^{\tilde{\kappa}}x}^p$ are bounded, we can successively compare and remove coefficients of the formulae for $Q_k(T^m_{j_0}\left(T_{j_0}'\right)^{\tilde{\kappa}}x)$ as given through Corollary \ref{Coro. Newtonform of (m,p) (i) and P_(m-1) (ii)} and \eqref{quasi-polynomial for k}, until we eventually obtain that
\[ \norm{T_{j_0}^{m+k}\left(T_{j_0}'\right)^{\tilde{\kappa}}x}^p = \norm{T^m_{j_0}\left(T_{j_0}'\right)^{\tilde{\kappa}}x}^p, \]
for all $k \in \N$, for all $x \in X$. 
That is, the sequences
\begin{align*}
	 n \mapsto \norm{T^n_{j_0}\left(T_{j_0}'\right)^{\tilde{\kappa}}x}
\end{align*}
become constant for $n \geq m$, for all $x \in X$. This concludes the induction step and the proof. 
\end{proof}

We can now prove the main result.

\begin{proof}[Proof of Theorem \ref{main theorem of paper}]
By Corollary \ref{Cor. structure of Q^n if (m,p) and (mu,infty)} and the lemma above, we have for $n \geq 2m-1$, 
\begin{align}\label{quasi-polynomial n}
	Q^n(x)=&\sum_{\substack{\beta \in \N^{d-1} \\ |\beta|=0,...,m-1}} n^{(|\beta|)} \sum_{j=1}^d 
	\frac{1}{\beta!} \norm{T_j^{n-|\beta|}(T'_j)^{\beta}x}^p \notag \\
	=&\sum_{\substack{\beta \in \N^{d-1} \\ |\beta|=1,...,m-1}} n^{(|\beta|)} \sum_{j=1}^d 
	\frac{1}{\beta!} \norm{T_j^{m}(T'_j)^{\beta}x}^p + \sum_{j=1}^d \norm{T^n_jx}^p,
 	\ \ \forall x \in X.
\end{align}
That is, for all $x \in X$, for $n \geq 2m-1$, the sequences $n \mapsto Q^n(x)$ become almost polynomial (of degree $\leq m-1$), with the term $\sum_{j=1}^d \norm{T^n_jx}^p$ instead of a (constant) trailing coefficient.

Again, by Theorem \ref{Theorem fundamental prop. (m,p)}, we know that for any $x \in X$, the sequences $n \mapsto Q^n(x)$ are indeed polynomial. And since, by Theorem \ref{Prop. norm(T^alpha x) is bounded}, for each $x \in X$, the sequences $n \mapsto \sum_{j=1}^d \norm{T^n_jx}^p$ are bounded, we can again successively compare and remove coefficients of the formulae for $Q_n(x)$ as given in Corollary \ref{Coro. Newtonform of (m,p) (i) and P_(m-1) (ii)} and \eqref{quasi-polynomial n}, until we eventually obtain that
\begin{align}\label{equation (1,p) for n geq 2m-1}
	\sum_{j=1}^d \norm{T^n_jx}^p = \norm{x}^p, \ \ \forall x \in X, \ \forall \ n \geq 2m-1 \ .
\end{align}

Since $T_i^mT_j^m = 0$ for all $i \neq j$,  by Lemma \ref{Lemma T_i^mT_j^m=0}, replacing $x$ by $T^{\nu}_{j}x$ 
with $\nu \geq m$ in this last equation, 
gives $\norm{T^{\nu}_{j}x} = \norm{T_{j}^{n+\nu}x}$ for all $n \geq 2m-1$, for all $x \in X$.  
Hence, the sequences $n \mapsto \norm{T_j^{n}x}$ become constant for $n \geq m$, for all
$j \in \{1,...,d\}$, for all $x \in X$. This is \ref{main theorem of paper}.(i).

But then, \eqref{equation (1,p) for n geq 2m-1} becomes
\begin{align*}
	\sum_{j=1}^d \norm{T^{m}_jx}^p = \norm{x}^p, \ \ \forall x \in X\ .
\end{align*}
This is \ref{main theorem of paper}.(ii).

Now take any $(n_1,...,n_d) \in \N^d$ with $n_j \geq m$ for all $j$ and replace $x$ in the equation above by 
$\sum_{j=1}^dT_j^{n_j}$. Then, again, since $T^m_iT_j^m=0$ for $i \neq j$, and since $n \mapsto \norm{T_j^{n}x}$ become constant for $n \geq m$,
\begin{align*}
	\sum_{j=1}^d \norm{T_j^{m+n_j}x}^p 
	= \sum_{j=1}^d \norm{T_j^{m}x}^p = \big \|\sum_{j=1}^d T_j^{n_j} x \big \|^p, \ \ \forall x \in X. 
\end{align*} 
Together with \ref{main theorem of paper}.(i), this implies \ref{main theorem of paper}.(iii). 
\end{proof}

\begin{corollary}
If one of the operators $T_{j_0} \in \{T_1,...,T_d\}$ is surjective, then Theorem \ref{main theorem of paper}.(i) forces this operator to be an isometric isomorphism and by \ref{main theorem of paper}.(ii) the remaining operators are nilpotent. 

If one of the operators $T_{j_0} \in \{T_1,...,T_d\}$ is injective, by Lemma \ref{Lemma T_i^mT_j^m=0} and \ref{main theorem of paper}.(ii) we obtain that $T^m_{j_0}$ is an isometry and the remaining operators are nilpotent.
\end{corollary}

However, with respect to the second part of this corollary, note that while, by definition of an $(m,p)$-isometry, we must have $\bigcap_{j=1}^d N(T_j) = \{0\}$, it is not clear that the kernel of a single operator has to be trivial.

\section{Some further remarks and the case $d=2$}

We finish this note with a stronger result for the case of a commuting pair $(T_1,T_2) \in B(X)^2$. We first state the following two easy corollaries of Theorem \ref{main theorem of paper} which hold for general $d$.

\begin{corollary}
Let $T=(T_1,...,T_d) \in B(X)^d$ be an $(m,p)$-isometry as well as a $(\mu,\infty)$-isometry. Then $T_j^m = 0$ or $\norm{T_j^m}=1$ for any 
$j\in \{1,...,d\}$.
\end{corollary}
\begin{proof}
By Theorem \ref{main theorem of paper}.(ii) we have $\norm{T_j^m} \leq 1$ for any $j$. On the other hand, by \ref{main theorem of paper}.(i) we have 
\begin{equation*}\norm{T_j^mx} = \norm{T_j^{m+1}x} \leq \norm{T_j^m} \cdot \norm{T_j^mx}, \ \ \forall x \in X,\end{equation*}
for any $j$. That is, $T_j^m = 0$ or $\norm{T_j^m} \geq 1$. 
\end{proof}

\begin{lemma}
Let $T=(T_1,...,T_d) \in B(X)^d$ be an $(m,p)$-isometry as well as a $(\mu,\infty)$-isometry. Define $|.|_{\infty}: X \to [0,\infty)$ and $X_{j, |.|_\infty}$ as in Theorem \ref{Th. fundamental prop. and partition for (m,infty)}. Then 
\begin{align*}
X_{j,|.|_{\infty}} = \{x \in X \ | \ &\exists \alpha(x) \in \N^{d}, \ \textrm{s.th.} \ |\alpha(x)| \leq \mu-1 \ \textrm{and} \\ 
&|x|_{\infty}=\norm{T^n_j\left(T'_j\right)^{\alpha'_j(x)}x}, \ \forall n \in \N \}.
\end{align*}
\end{lemma} 
\begin{proof}
By Theorem \ref{Prop. norm(T^alpha x) is bounded} we know that for every $x \in X$, there exists an $\alpha(x) \in \N^d$ with 
$\max_{\alpha \in \N^d}\norm{T^\alpha x} = \norm{T^{\alpha(x)}x}$ and $|\alpha(x)| \leq \mu-1$.

Then $x \in X_{j,|.|_\infty}$ if, and only if, for all $n \in \N$, there exists an $\alpha(x,n) \in \N^d$ with $|\alpha(x,n)| \leq \mu-1$
s.th. $|x|_{\infty}=\norm{T^n_jT^{\alpha(x,n)}x}$. Hence, the inclusion ``$\supset$" is clear.

To show ``$\subset$" let $0 \neq x \in X_{j,|.|_{\infty}}$. Then $T^m_j \neq 0$ and, hence, $\norm{T_j^m}=1$.

Since $|\alpha(x,n)| \leq \mu-1$ for all $n \in \N$, there are only finitely many choices for each $\alpha(x,n)$. Thus, there exists an 
$\alpha(x) \in \N^d$ and an infinite set $M(x) \subset \N$ s.th.
\begin{equation*}|x|_{\infty}=\norm{T^n_jT^{\alpha(x)}x}, \ \forall n \in M(x).\end{equation*}
By Theorem \ref{main theorem of paper}.(i), $M(x)$ contains all $n \geq m$ and further,
\begin{equation*}\norm{T^n_jT^{\alpha(x)}x} = \norm{T^n_j\left(T'_j\right)^{\alpha'_j(x)}x}, \ \ \textrm{for all} \ n \geq m.\end{equation*}
Since $\norm{T_j^m}=1$, the statement holds for all $n \in \N$.
\end{proof}

\begin{proposition}
Let $T=(T_1,T_d) \in B(X)^2$ be both an $(m,p)$-isometric and a $(\mu,\infty)$-isometric pair. Then $T^m_1$ is an isometry and $T^m_2=0$ or vice versa.
\end{proposition}
\begin{proof}
By Theorem \ref{Th. fundamental prop. and partition for (m,infty)}, we have $X = X_{1,|.|_{\infty}} \cup X_{2,|.|_{\infty}}$. 

Let $x_1 \in X_{1,|.|_{\infty}}$. Then, by the previous lemma, there exists an $\alpha_2(x_1) \in \N$ with $\alpha_2(x_1) \leq \mu-1$ s.th. $|x_1|_{\infty} = \norm{T_1^nT_2^{\alpha_2(x_1)}x_1}$ for all $n \in \N$.

Furthermore, we have $\norm{x}^p = \norm{T_1^mx}^p+\norm{T_2^mx}^p$, for all $x \in X$, by Theorem \ref{main theorem of paper}.(ii). Replacing $x$ by $T_2^{\alpha_2(x_1)}x_1$gives
\begin{align*}
&\norm{T_2^{\alpha_2(x_1)}x_1}^p = \norm{T_1^mT_2^{\alpha_2(x_1)}x_1}^p + \norm{T_2^{m + \alpha_2(x_1)}x_1}^p \\
\Leftrightarrow \ \ &\norm{T_2^{\alpha_2(x_1)}x_1}^p = |x_1|^p_{\infty} + \norm{T_2^{m}x_1}^p.
\end{align*}
This implies $\norm{T_2^{\alpha_2(x_1)}x_1} = |x_1|_{\infty}$ and, moreover, $\norm{T_2^{m}x_1}=0$.

An analogous argument shows that $X_{2,|.|_{\infty}} \subset N(T^m_1)$. Hence, 
\begin{equation*}X = N(T_1^m) \cup N(T_2^m),\end{equation*} 
which forces $T_1^m = 0$ or $T^m_2 = 0$. The statement follows from $\norm{x}^p = \norm{T_1^mx}^p+\norm{T_2^mx}^p$, for all $x \in X$. 
\end{proof}

\bigskip
\hrule
\
\\

\begin{quote}
Philipp Hoffmann\\[1.5ex] 
Keywords International Ltd.\\
Philips House\\
South County Business Park\\
Dublin 18\\[1.5ex] 
\begin{tabular}{@{}l@{ }l}
\emph{email:} & \url{philipp.hoffmann@maths.ucd.ie}
\end{tabular}
\end{quote}


\begin{thebibliography}{99}
\bibitem{Ag} J. Agler, A disconjugacy theorem for Toeplitz operators, \emph{Am. J. Math.}, Vol. 112. No. 1 (1990), 1-14.
\bibitem{AgStanI} J. Agler and M. Stankus, $m$-isometric transformations of Hilbert space, I, \emph{Integr. equ. oper. 
theory}, Vol. 21, No. 4 (1995), 383-429.
\bibitem{Bay} F. Bayart, $m$-Isometries on Banach Spaces, \emph{Mathematische Nachrichten}, Vol. 284, No. 17-18 (2011), 
2141-2147.
\bibitem{BeMaMu} T. Bermúdez, A. Martinón and V. Müller, $(m,q)$-isometries on metric spaces, \emph{J. Operator Theory}, Vol. 72, No. 2 (2014), 313-329.
\bibitem{BeMaNe} T. Bermúdez, A. Martinón and E. Negrín, Weighted Shift Operators Which are m-Isometries,
\emph{Integr. equ. oper. theory}, Vol. 68, No. 3 (2010), 301-312.
\bibitem{Bo} F. Botelho, On the existence of $n$-isometries on $\ell_p$ spaces, \emph{Acta Sci. Math. (Szeged)}, Vol. 76, No. 1-2 (2010), 183-192.
\bibitem{HoMa} P. H. W. Hoffmann and M. Mackey, $(m,p)$-isometric and $(m,\infty)$-isometric operator tuples on normed spaces,
\emph{Asian-Eur. J. Math.}, Vol. 8, No. 2 (2015).
\bibitem{HoMaOS} P. Hoffmann, M. Mackey and M. Ó Searcóid, On the second parameter of an $(m,p)$-isometry, \emph{Integr. 
equ. oper. theory}, Vol. 71, No. 3 (2011), 389-405.
\bibitem{GleaRi} J. Gleason and S. Richter, $m$-Isometric Commuting Tuples of Operators on a Hilbert Space,
\emph{Integr. equ. oper. theory}, Vol. 56, No. 2 (2006), 181-196 .
\bibitem{Ri} S. Richter, Invariant subspaces of the Dirichlet shift, \emph{J. reine angew. Math.}, Vol. 386 (1988), 205-220.
\bibitem{SidAh} O.A. Sid Ahmed, $m$-isometric Operators on Banach Spaces, \emph{Asian-Eur. J. Math.}, Vol. 3, No. 1 (2010), 1-19.
\end{thebibliography}
\end{document}